\documentclass{amsart}
\usepackage{amssymb}
\usepackage{amsfonts}

\newtheorem{theorem}{Theorem}
\theoremstyle{plain}

\newtheorem{corollary}{Corollary}

\newtheorem{definition}{Definition}

\newtheorem{proposition}{Proposition}
\newtheorem{remark}{Remark}

\numberwithin{equation}{section}
\input{tcilatex}

\begin{document}
\title{Solving the noncommutative Batalin-Vilkovisky equation.}
\author{Serguei Barannikov\vspace{-0.3cm}}
\address{ENS(Paris), France}
\email{sergueibar@gmail.com, barannik@ens.fr}

\begin{abstract}
Given an odd symmetry acting on an associative algebra, I show that the
summation over arbitrary ribbon graphs gives the construction of the
solutions to the noncommutative Batalin-Vilkovisky equation, including the
equivariant version, introduced in my previous papers\cite{B06a},\cite{B06b}%
. This generalizes the known construction of $A_{\infty }-$algebra via
summation over ribbon trees. These solutions provide naturally the
supersymmetric matrix action functionals, which are the $gl(N)$%
-equivariantly closed differential forms on the matrix spaces, as described
in \cite{B06b}, \cite{B09a}.
\end{abstract}

\maketitle

\section{\protect\bigskip Introduction.\protect\footnotetext[1]{%
Electronic CNRS\ preprint hal-00464794 (17/03/2010)}}

I prove that the summation over arbitrary ribbon graphs with legs produces
explicit solutions to the noncommutative Batalin-Vilkovisky equation, which
I have introduced in \cite{B06a}. This generalizes the construction of $%
A_{\infty }-$structure via summation over trees, see \cite{K}, \cite{M98},%
\cite{H} and references therein.

The noncommutative Batalin-Vilkovisky equation is the equation 
\begin{equation}
\hbar \Delta S+\frac{1}{2}\{S,S\}=0\Leftrightarrow \Delta (\exp \frac{1}{%
\hbar }S)=0  \label{ncBV}
\end{equation}%
for elements of symmetric product of cyclic words 
\begin{equation*}
S=\sum_{i,g}\hbar ^{2g+i-1}S_{i,g},~~S_{i,g}\in Symm^{i}(\oplus
_{j=1}^{\infty }((\Pi B)^{\otimes j})^{\mathbb{Z}/j\mathbb{Z}})^{\vee }
\end{equation*}%
where $B$ is a $\mathbb{Z}/2\mathbb{Z}$-graded vector space with odd scalar
product, $\Delta $ is the odd second order operator, defined via
dissection-gluing of cyclic cochains see section 1.2 in \cite{B06b} or
section 5 in \cite{B06a}. For $B$ with even scalar product the
noncommutative Batalin-Vilkovisky equation and the operator $\Delta $ are
defined on elements of \emph{exterior} product of cyclic words $Symm(\oplus
_{j=1}^{\infty }\Pi ((\Pi B)^{\otimes j})^{\mathbb{Z}/j\mathbb{Z}})^{\vee }$.

I've associated in \cite{B06b} to any solution $S$ to equation (\ref{ncBV})
the $A-$infinity $gl(N)-$equivariant matrix integral of the form, in the odd
scalar product case, 
\begin{equation*}
\int_{\gamma }\exp \frac{1}{\hbar }(Tr(m_{A_{\infty }})+\sum_{2g+i>1}\hbar
^{2g+i-1}S_{i,g}(X)+\frac{1}{2}\left\langle [\Xi ,X],X\right\rangle )w~dX
\end{equation*}%
where $m_{A_{\infty }}=S_{1,0}$ is the term of $S$ of the lowest order in $%
\hbar $, which is the cyclic $A_{\infty }-$algebra tensor. The term $\exp 
\frac{1}{\hbar }(\sum_{2g+i>1}\hbar ^{2g+i-1}S_{i,g}+\frac{1}{2}\left\langle
[\Xi ,X],X\right\rangle )$ can be understood as a multitrace and equivariant
completion of $\exp (\frac{1}{\hbar }Tr(m_{A_{\infty }}))$ to a closed $%
gl(N)-$equivariant differential form, see \cite{B09a}. This is the
integration framework which I've proposed to associate with the equation $%
\{m_{A_{\infty }},m_{A_{\infty }}\}=0$.

These integrals can be understood as the generalisation of the matrix Airy
integral to arbitrary higher dimensions and, simulteneously, as the
noncommutative generalisation of periods of Calabi-Yau manifolds. They have
many remarkable properties, see \cite{B06b}, \cite{B09a}, \cite{B09c}. For
example, I've proven in \cite{B06b} that their asymptotic expansion is given
by the pairing between the characteristic classes $c_{S}\in H_{\ast }(%
\overline{\mathcal{M}}_{g,n})$ of the quantum $A_{\infty }-$algebra and the
classes $c_{\Lambda }\in H^{\ast }(\overline{\mathcal{M}}_{g,n})$ associated
with the odd part of the equivariant $gl(N)-$action. The solutions to the
equation (\ref{ncBV}) are the lagrangians for these matrix integrals and
this is one of the principal reasons why it is important to develop the
instrument, which constructs explicitely and classifies such solutions.

The universal look of the noncommutative Batalin-Vilkovisky equation and the
property that it leads to the remarkable integration theory, suggests that
the $A_{\infty }-$enhancement, widely used in the non-commutative derived
algebraic geometry, should in many interesting cases be considered as the
first order approximation to the full structure described by the solution to
the noncommutative Batalin-Vilkovisky equation in the space of symmetric or
exterieur powers of cyclic tensors. I explain below how to write down such
full structure in this context, see the corollary \ref{corwell}.

I work from the beginning in the equivariant extension of the noncommutative
Batalin-Vilkovisky formalism \cite{B09a}, \cite{B06b}. In particular my
results in this paper are valid in the framework of algebraic structures
with supersymmetry. My basic setting throughout the paper is an odd linear
operator $I$, which acts as a symmetry of associative algebra $A$, or more
general algebraic structure, and whose square is, in general, nonzero $%
I^{2}\neq 0$. The interesting applications include both the cases with $%
I^{2}=0$ and with $I^{2}\neq 0$. An important example from \cite{B06b},
deeply related with tauthological classes on the moduli space of Riemann
sufaces, is related with the odd differentiation $[\Xi ,\cdot ]$ acting on
the Bernstein-Leites odd general matrix algebra $q(N)$ with its odd trace,
where $\Xi $ is an odd matrix from $q(N)$. In this case $I^{2}=[\Xi
^{2},\cdot ]$ and $I^{2}\neq 0$.

Briefly, my construction is a sum, over ribbon graphs with legs, of the
tensors $W_{\Gamma }$, defined on symmetric/exterior products of cyclic
words, and given by the contraction defined by the ribbon graph, where the
elements of $\Pi B$ are attached to the legs of the ribbon graph, the
structure constants of the algebra $A$ (associative, $A_{\infty }$) are
attached to the vertices and the propagator, which is the inverse to the
scalar product modified by homotopy inverse to $I$, is attached to the
edges, see (\ref{wg}). I give also the generalizations for summation over
stable ribbon graphs, whose relation with the noncommutative
Batalin-Vilkovisky equation was described in \cite{B06a}.

The construction here is very closely related with my construction from \cite%
{B06a}. In \cite{B06a} I've constructed the homology class in the stable
ribbon graph complex from any solution of the noncommutative
Batalin-Vilkovisky equation. The construction there associates the number to
any stable ribbon graph with no legs. It is the contraction, with the
products of certain multitrace tensors $m_{i,g}$  attached to the vertices,
and the inverse to the scalar product attached to the edges. Here the
ingredients in the construction giving solution to the Batalin-Vilkovisky
equation are similar. Except that here I relax the condition on the linear
term: $d^{2}\neq 0$. To avoid the confusion this linear term is denoted by $I
$. In addition I take a  self-adjoint operator $H$ such that $Id-[I,H]=P$ is
idempotent, $[I^{2},H]=0$, notice that $H^{2}\neq 0$ in general. In a sense
it is a homotopy inverse to $I$ on the subspace $Id-P$. I use $H$ to modify
the scalar product associate with edges. The outcome of the construction is
a solution, on the image of the idempotent $P$, to the noncommutative
Batalin-Vilkovisky equation (equivariant if $I^{2}$ is nonzero on the image
of $P$).

The similarity of the construction in this paper and the construction from
the previous paper \cite{B06a} is not accidental. In fact, the
quasiisomorphism $PA\hookrightarrow A$ induces the quasi-isomorphism from
the dg-Lie algebra $\underline{Mor}(\mathbb{M}_{\mathcal{D}^{\vee }}\mathcal{%
P}^{dual},\mathcal{E}[A])$ to $\underline{Mor}(\mathbb{M}_{\mathcal{D}^{\vee
}}\mathcal{P}^{dual},\mathcal{E}[PA]),$ see \cite{B06a}, section 5.1. And
the construction described below is the result of the action of this
quasi-isomorphism on the solution to the noncommutative Batalin-Vilkovisky
equation. On the other hand such solution is the same, by the theorem 1 from 
\cite{B06a}, as the action of the stable ribbon graphs complex, whose part
on graphs with no legs gives the homology classes construction from \cite%
{B06a}.

It is important to stress that in order that the summation over ribbon
graphs, with the $A_{\infty }-$tensors attached to vertices, gives the
solution to the noncommutative Batalin-Vilkovisky equation, the following
extra condition must be imposed on the cyclic $A_{\infty }-$algebra 
\begin{equation*}
\Delta m_{A_{\infty }}=0.
\end{equation*}%
One significant case when this condition is automatically satisfied is the
case of the $\mathbb{Z}/2\mathbb{Z}$-graded associative algebra. Another
possibility to construct the solution on the subspace $PA\subset A$ is to
extend the $A_{\infty }-$structure on $A$ to a solution to the
noncommutative Batalin--Vilkovisky equation in the initial space $A$ and
then take the sum over \emph{stable} ribbon graphs. I describe the case when
such extension is possible in the corollary\nolinebreak\ \ref{corwell}.

Here is the brief description of the content of the paper. In section \ref%
{sAss} I describe the summation over ribbon graphs starting from the input
data given by the odd symmetry $I$ acting on the $\mathbb{Z}/2\mathbb{Z}$%
-graded associative algebra $A$, $\dim _{k}A<\infty $ , with odd or even
scalar product. This is the main construction of the paper. Next I consider
the case of cyclic $A_{\infty }-$algebra. 

The general case of associative algebra with no scalar product, and also $%
A_{\infty }-$algebra with no scalar product, is then reduced to the case of
the algebra with scalar product by putting $\widetilde{A}=A\oplus (\Pi
A)^{\vee }$ with odd scalar product given by natural odd pairing between $A$
and $(\Pi A)^{\vee }$, or by putting $\widetilde{A}=A\oplus A^{\vee }$ with
even scalar product given by natural even pairing between $A$ and $A^{\vee }$%
. 

In the section \ref{sectmodop} I give the generalization to the case of
summation over graphs with arbitrary $\mathbb{S}_{n}-$modules, endowed with
some contraction operations, which are attached to the vertices. The
supersymmetry $I$ in this case is acting on an algebra  over $\mathcal{FP}$,
the Feynman transform of an  arbitrary twisted modular operad $\mathcal{P}$,
and the summation over $\mathcal{P}-$marked graphs  gives the solution  to
the $\mathcal{P}-$type Batalin-Vilkovisky equation introduced in   \cite%
{B06a}. 

In the last section \ref{sectionIntegr} I outline the applications of the
construction by recalling the relations from \cite{B06b}, \cite{B09a} of the
noncommutative Batalin-Vilkovisky formalism with equivariant matrix
integration, in particular the correspondence between solutions to the
noncommutative Batalin-Vilkovisky equation and equivariantly closed
differential forms on $\left( gl(N|N)\otimes \Pi V\right) _{0}$. 

\emph{Notations}. I work in the tensor category of super vector spaces, over
an algebraically closed field $k$, $char(k)=0$. Let $V_{0}\oplus V_{1}$ be a 
$\mathbb{Z}/2\mathbb{Z}$-graded vector space. I denote by $\overline{\alpha }
$ the parity of an element $\alpha $ and by $\Pi V$ the super vector space
with inversed parity. Element $(a_{1}\otimes a_{2}\otimes \ldots \otimes
a_{n})$ of $A^{\otimes n}$ is denoted by $(a_{1},a_{2},\ldots ,a_{n})$. I
denote by $V^{\vee }$ the dual vector space $\limfunc{Hom}(V,k)$. For a
module $U$ over a finite group $G$ I denote via $U^{G}$ the subspace of
invariants: $\{\forall g\in G:gu=u|u\in U\}$. A graph $\Gamma $ is a triple $%
(Flag(\Gamma ),\lambda ,\sigma )$, where $Flag(\Gamma )$ is a finite set,
whose elements are called flags, $\lambda $ is a partition of $Flag(\Gamma )$%
, and $\sigma $ is an involution acting on $Flag(\Gamma )$. By partition
here one understands a disjoint decomposition into unordered subsets. These
subsets are the vertices of the graph. The set of vertices is denoted by $%
Vert(\Gamma )$. The subset of $Flag(\Gamma )$ corresponding to vertex $v$ is
denoted by $Flag(v)$. The cardinality of $Flag(v)$ is called the valence of $%
v$ and is denoted $n(v)$. The edges of the graph are the pairs of flags
forming a non-trvial two-cycle of the involution $\sigma $. The set of edges
is denoted $Edge(\Gamma )$. The legs of the graph are the fixed elements of
the involution $\sigma $. The set of legs is denoted $Leg(\Gamma )$. The
number of legs is denoted $n(\Gamma )$. The cardinality of a finite set $X$
is denoted by $|X|$. Throughout the paper, unless it is stated explicitely
otherwise, $(-1)^{\epsilon }$ in the formulas denotes the standard Koszul
sign, which can be worked out by counting $(-1)^{\overline{a}\overline{b}}$%
every time the objects $a$ and $b$ are interchanged to obtain the given
formula.

\section{\protect\bigskip\ Associative algebra with odd differentiation. 
\label{sAss}}

I consider in this section a $\mathbb{Z}/2\mathbb{Z}-$graded associative
algebra $A$, $\dim _{k}A<\infty $ , with multiplication denoted by $%
m_{2}:A^{\otimes 2}\rightarrow A$ and an \emph{odd differentiation} $%
I:A\rightarrow \Pi A$%
\begin{equation*}
Im_{2}(a,b)=m_{2}(Ia,b)+(-1)^{\overline{a}}m_{2}(a,Ib),
\end{equation*}%
in particular, if $I^{2}=0$ then this is a d$(\mathbb{Z}/2\mathbb{Z)}$%
g-algebra.

\subsection{\protect\bigskip Odd scalar product.}

I assume that the algebra $A$,  is cyclic with odd scalar product 
\begin{equation*}
\beta :A^{\otimes 2}\rightarrow \Pi A,
\end{equation*}%
so that the three tensor 
\begin{eqnarray*}
m &\in &((\Pi A)^{\otimes 3})^{\vee } \\
m(\pi a,\pi b,\pi c) &=&(-1)^{\overline{b}}\beta (m_{2}(a,b),c)
\end{eqnarray*}%
is cyclically invariant 
\begin{equation*}
m(\pi a,\pi b,\pi c)=(-1)^{(c+1)(a+b)}m(\pi c,\pi a,\pi b)
\end{equation*}%
and that $\beta $ is preserved by $I$:%
\begin{equation*}
\beta (Ia,b)+(-1)^{\overline{a}}\beta (a,Ib)=0.
\end{equation*}%
The modification for the variant with an even scalar product are described
below.

Below I consider also the variant for general d$(\mathbb{Z}/2\mathbb{Z})$%
g-algebra without scalar product. It is reduced to the case with even/odd
scalar product by putting $\widetilde{A}=A\oplus A^{\vee }$, or $\widetilde{A%
}=A\oplus \Pi A^{\vee }$ with their natural scalar products.

I have 
\begin{equation}
m(Ia,b,c)+(-1)^{\overline{a}}m(a,Ib,c)+(-1)^{\overline{a}+\overline{b}%
}m(a,b,Ic)=0  \label{Im}
\end{equation}%
which reflect the Leibnitz rule for the differentiation $I$. Denote by $%
\beta ^{\vee }\in (\Pi A)^{\otimes 2}$ the tensor of the scalar product on
the dual vector space, then for any $a,b,c,d\in A$, 
\begin{equation}
\left\langle m(\pi a,\pi b,\cdot )m(\cdot ,\pi c,\pi d),\beta ^{\vee
}\right\rangle =(-1)^{\varepsilon }\left\langle m(\pi d,\pi a,\cdot )m(\cdot
,\pi b,\pi c),\beta ^{\vee }\right\rangle  \label{mm}
\end{equation}%
which is the associativity of the multiplication $m$.

\bigskip Let $H$ be an odd selfadjoint operator%
\begin{equation*}
H:A\rightarrow \Pi A,~~~H^{\vee }=H
\end{equation*}%
such that 
\begin{equation}
Id-[I,H]=P  \label{Idp}
\end{equation}%
is an idempotent operator $P:A\rightarrow A,$%
\begin{equation*}
~P^{2}=P.
\end{equation*}%
I assume also that $H$ commutes with $I^{2}$, this is automatic if $I^{2}=0$%
. Such $H$ can be always be found for example by considering the kernel $%
\ker I^{2}=\{x|I^{2}x=0\}$, on which $H$ is a homotopy on the complement to
a space representing the cohomology of $I|_{\ker I^{2}}$, and the orthogonal
complement $(\ker I^{2})^{\intercal }$ of $\ker I^{2}$ on which $I^{2}$ is
invertible and on which $H$ can be taken for example to be $\frac{1}{2}%
I|_{(\ker I^{2})^{\intercal }}^{-1}$. Notice that in general $H^{2}\neq 0$.
I denote by $B$ the subspace which is the image of the idempotent $P$.

Let $\Gamma $ be a tri-valent ribbon graph, i.e. the trivalent graph with
fixed cyclic orders on the sets of the three flags attached to every vertex.
Let $\Sigma _{\Gamma }$ be the corresponding oriented two-dimensional
surface. Then I put:

\begin{itemize}
\item \bigskip the three-tensors 
\begin{equation*}
m^{v}\in ((\Pi A)^{\otimes Flag(v)})^{\vee }
\end{equation*}%
on every vertex $v$

\item the two tensors 
\begin{gather*}
\beta _{H}^{\vee ,e}\in (\Pi A)^{\otimes \{f,f^{\prime }\}}, \\
\beta _{H}^{\vee ,e}=\beta ^{\vee }(H^{\vee }u_{f},v_{f^{\prime }})=(-1)^{%
\overline{u_{f}}~\overline{v_{f^{\prime }}}}\beta ^{\vee }(H^{\vee
}v_{f^{\prime }},u_{f})
\end{gather*}%
for any interieur edge $e=(ff^{\prime })$

\item element $a_{l}\in \Pi B$, for any leg $l\in Leg(\Gamma ),$ this gives
a partition of the set of elements $\{a_{l}\}_{l\in Leg(\Gamma )}$ to the
subsets corresponding to the components of the boundary $\partial \Sigma
_{\Gamma }$ and the cyclic orders on these subsets.
\end{itemize}

Notice that both $m^{v}$ and $\beta _{H}^{\vee ,e}$ are even elements, so
that the products 
\begin{equation*}
\tbigotimes_{v\in Vert(\Gamma )}m^{v}\in ((\Pi A)^{\otimes Flag(\Gamma
)})^{\vee }
\end{equation*}
and 
\begin{equation*}
\tbigotimes_{e\in Edge(\Gamma )}\beta _{H}^{\vee ,e}\in (\Pi A)^{\otimes
Flag(\Gamma )\setminus Leg(\Gamma )}
\end{equation*}
are canonically defined.

\begin{definition}
I define the tensor $W_{\Gamma }$ as the contraction 
\begin{equation}
W_{\Gamma }(\tbigotimes_{l\in Leg(\Gamma )}a_{l})=\left\langle
\tbigotimes_{v\in Vert(\Gamma )}m^{v},\left( \tbigotimes_{e\in Edge(\Gamma
)}\beta _{H}^{\vee ,e}\right) \tbigotimes_{l\in Leg(\Gamma
)}a_{l}\right\rangle  \label{wg}
\end{equation}
\end{definition}

Notice that $W_{\Gamma }$ is cyclically invariant on every subset of $%
\{a_{l}\}_{l\in Leg(\Gamma )}$ corresponging to a component of the boundary
of $\Sigma _{\Gamma }$. Moreover the cyclic orders on flags at vertices
induce the orientation on the ribbon graph $\Gamma $, whose detailed
analysis, see e.g.\cite{B09b}, \cite{B06a} or section \ref{sectmodop} below,
shows that $W_{\Gamma }$ belongs to the symmetric product 
\begin{equation*}
W_{\Gamma }\in Symm(\oplus _{j=1}^{\infty }(\Pi B^{\otimes j})^{\mathbb{Z}/j%
\mathbb{Z}})^{\vee }
\end{equation*}%
Let $\chi (\Sigma _{\Gamma })$ denotes the genus of $\Sigma _{\Gamma }$, 
\begin{equation*}
\chi (\Sigma _{\Gamma })=2-2g(\Sigma _{\Gamma })-i(\Sigma _{\Gamma }),
\end{equation*}%
where $g(\Sigma _{\Gamma })$, $i(\Sigma _{\Gamma })$ are the genus and the
number of boundary components of $\Sigma _{\Gamma }$. I put 
\begin{equation}
S=\tsum_{\{\Gamma \}}\hbar ^{1-\chi (\Sigma _{\Gamma })}W_{\Gamma }
\label{S}
\end{equation}%
where the sum is over isomorphism classes of connected trivalent graphs with
nonempty subsets of legs on every boundary component of $\Sigma _{\Gamma }$.
One can include the graphs with empty subsets of legs on boundary components
by adding the constant term to the Batalin-Vilkovisky operator $\Delta $, I
leave the details to the interested reader.

\begin{proposition}
The number of such ribbon graphs with fixed $\chi (\Sigma _{\Gamma })$ and
fixed number of $n(\Gamma )$ of legs is finite.
\end{proposition}

\begin{proof}
This is a standard lemma, whose proof I include here for convenience of the
reader. The number of flags gives%
\begin{equation*}
n(\Gamma )+2|Edge(\Gamma )|=3|Vert(\Gamma )|
\end{equation*}%
Also 
\begin{equation*}
\chi (\Sigma _{\Gamma })=|Vert(\Gamma )|-|Edge(\Gamma )|
\end{equation*}%
since $\Sigma _{\Gamma }$ is homotopic to the geometric realisation of $%
\Gamma $. It follows that 
\begin{equation*}
|Edge(\Gamma )|=n(\Gamma )-3\chi (\Sigma _{\Gamma })
\end{equation*}%
and 
\begin{equation*}
|Vert(\Gamma )|=n(\Gamma )-2\chi (\Sigma _{\Gamma })
\end{equation*}%
and hence the number of such graphs is finite.
\end{proof}

\begin{theorem}
\label{th1}The sum over ribbon graphs $S$ defined in (\ref{S}) satisfy the
equivariant noncommutative Batalin-Vilkovisky equation: 
\begin{equation}
\hbar \Delta S+\frac{1}{2}\{S,S\}+I^{\vee }S=0,\,\,  \label{eqBV}
\end{equation}%
in particular if $I|_{B}$ is zero then $S$ is the solution of the
non-commutative Batalin-Vilkovisky equation from \cite{B06a},\cite{B06b} 
\begin{equation*}
\hbar \Delta S+\frac{1}{2}\{S,S\}=0
\end{equation*}%
If $I|_{B}\neq 0$, but $I^{2}|_{B}=0$, then $S+$ $S_{0,2}$ is also a
solution to the non-commutative Batalin-Vilkovisky equation from \cite{B06a},%
\cite{B06b}, where $S_{0,2}=(-1)^{\epsilon }\beta (I\cdot ,\cdot )|_{B}$ is
the quadratic term corresponding to the differential $I|_{B}$.
\end{theorem}

\begin{proof}
The proof is straightforward. For a trivalent graph $\Gamma $ and an
internal edge $e\in Edge(\Gamma )$ consider the three tensors 
\begin{equation*}
W_{\Gamma ,e}^{[I,H]},W_{\Gamma ,e}^{Id},W_{\Gamma ,e}^{P}\in Symm(\oplus
_{j=1}^{\infty }((\Pi B)^{\otimes j})^{\mathbb{Z}/j\mathbb{Z}})^{\vee }
\end{equation*}%
which are defined by the same contraction as $W_{\Gamma }$ except that at
the edge $e\in Edge(\Gamma )$ I put the tensors 
\begin{equation*}
\beta ^{\vee }([I^{\vee },H^{\vee }]u_{f},v_{f^{\prime }}),~\beta ^{\vee
}(u_{f},v_{f^{\prime }}),\beta ^{\vee }(P^{\vee }u_{f},v_{f^{\prime }})
\end{equation*}%
correspondingly instead of $\beta _{H}^{\vee ,e}$. Then, from (\ref{Idp}) 
\begin{equation*}
W_{\Gamma ,e}^{P}=W_{\Gamma ,e}^{Id}-W_{\Gamma ,e}^{[I,H]}.
\end{equation*}%
By summing over $v\in Vert(\Gamma )$ of the Leibnitz rule (\ref{Im}) and
noticing that 
\begin{equation*}
\beta ^{\vee }(H^{\vee }I^{\vee }u_{f},v_{f^{\prime }})+\beta ^{\vee
}(u_{f},H^{\vee }I^{\vee }v_{f^{\prime }})=-\beta ^{\vee }([I^{\vee
},H^{\vee }]u_{f},v_{f^{\prime }})
\end{equation*}%
I get 
\begin{equation*}
I^{\vee }W_{\Gamma }-\tsum_{e}W_{\Gamma ,e}^{[I,H]}=0.
\end{equation*}%
Next I use (\ref{mm}) to substitute in $W_{\Gamma ,e}^{Id}$ the contraction 
\begin{equation*}
\left\langle m(\pi a,\pi b,\cdot )m(\cdot ,\pi c,\pi d),\beta ^{\vee
}\right\rangle 
\end{equation*}%
corresponding to the internal edge $e\in Edge(\Gamma )$ by 
\begin{equation*}
(-1)^{\varepsilon }\left\langle m(\pi d,\pi a,\cdot )m(\cdot ,\pi b,\pi
c),\beta ^{\vee }\right\rangle .
\end{equation*}%
This corresponds to passing from the trivalent ribon graph $\Gamma $ to the
trivalent ribbon graph $\Gamma ^{\prime }$ obtained by the standard
transformation on the edge $e$, preserving the overal cyclic order of the
flags corresponding to $\pi a$, $\pi b$, $\pi c$, $\pi d$. This
transformation preserves the surface $\Sigma _{\Gamma }$ and the
distribution of elements of $Leg(\Gamma )$ over the boundary components of $%
\Sigma _{\Gamma }$. Therefore the sum of $W_{\Gamma ,e}^{Id}$ over all
internal edges and over the set of trivalent graphs, having the same $\Sigma
_{\Gamma }$ with same distribution of $Leg(\Gamma )$ over the boundary
components, is zero: 
\begin{equation*}
\tsum_{\{\Gamma \},\Sigma _{\Gamma }=\Sigma ,e\in Edge(\Gamma )}W_{\Gamma
,e}^{Id}=0.
\end{equation*}%
Notice that $P^{2}=P$ implies that 
\begin{equation*}
\beta ^{\vee }(P^{\vee }u_{f},v_{f^{\prime }})=\beta ^{\vee }(P^{\vee
}u_{f},P^{\vee }v_{f^{\prime }}).
\end{equation*}%
Then, from the definition of the Batalin-Vilkovisky operator and the odd
Poisson bracket on $B$ it follows that 
\begin{multline*}
\tsum_{\{\Gamma \}}\hbar ^{2-\chi (\Sigma _{\Gamma })}\Delta W_{\Gamma }+%
\frac{1}{2}\{\tsum_{\{\Gamma \}}\hbar ^{1-\chi (\Sigma _{\Gamma })}W_{\Gamma
},\tsum_{\{\Gamma ^{\prime }\}}\hbar ^{1-\chi (\Sigma _{\Gamma ^{\prime
}})}W_{\Gamma ^{\prime }}\}= \\
=\tsum_{\{\widetilde{\Gamma }\},e\in Edge(\widetilde{\Gamma })}\hbar
^{1-\chi (\Sigma _{\Gamma })}W_{\widetilde{\Gamma },e}^{P}
\end{multline*}%
where each term on left hand side corresponds precisely to the right hand
side term $\hbar ^{1-\chi (\Sigma _{\Gamma })}W_{\widetilde{\Gamma },e}^{P}$%
, where $\widetilde{\Gamma }$ is obtained by gluing two legs to form the
edge $e$ from either the single surface or the two surfaces . Notice that
the condition, that $\Delta $ does not get contributions from the
neighboring points on the same circle, and that $\Delta $ and $\{\cdot
,\cdot \}$ do not get contributions from pair of cycles with just one
element on each, corresponds precisely to the fact that the resulting
surface $\Sigma _{\widetilde{\Gamma }}$ has always nonempty subsets of $Legs(%
\widetilde{\Gamma })$ on the boundary components.
\end{proof}

\begin{remark}
In the infinite dimensional case instead of $\beta ^{\vee }(H^{\vee }\cdot
,\cdot )$ the kernel constructed from appropriate resolution of the
diagonal, i.e. $A$ as $A^{op}\otimes A$-bimodule, must be used. Details will
appear elsewhere.  
\end{remark}

\begin{corollary}
\label{corwell}Given a $\mathbb{Z}/2\mathbb{Z-}$graded cyclic $A_{\infty }-$%
algebra $B$, if $B$ has cyclic d($\mathbb{Z}/2\mathbb{Z})$g associative
model $A$, in the sense that $B$ is obtained from $A$ via the summation
over\ trees, such that for $A$ the contractions (\ref{wg}) over arbitrary
trivalent ribbon graphs are well-defined, then the summation over such
graphs gives an extension of the cyclic $A_{\infty }-$algebra on $B$ to the
solution of the non-commutative Batalin-Vilkovisky equation.
\end{corollary}

\subsection{Even scalar product.}

Assume now that the scalar product on $A$ is even: 
\begin{equation*}
\beta :A^{\otimes 2}\rightarrow A.
\end{equation*}%
Then given an \emph{odd differentiation} $I:A\rightarrow \Pi A$ and an odd
selfadjoint operator $H:A\rightarrow \Pi A,~~~$satisfying (\ref{Idp}), I
construct the tensors $W_{\Gamma }$ for any ribbon trivalent graph by the
same contraction (\ref{wg}). The only difference is that in this case, both
the three-tensors $m^{v}$ attached to the vertices and the two-tensors $%
\beta _{H}^{\vee ,e}$ are odd and I sum over oriented ribon graphs where the
orientation is an orientation on the space $k^{Flag(\Gamma )}$. Carefull
analysis of the corresponding orientation on $\Gamma $, analogous to the one
from \cite{B06a}, shows that $W_{\Gamma }$ belongs to the \emph{exterior}
power of the space of cyclic tensors%
\begin{equation*}
W_{\Gamma }\in Symm(\oplus _{j=1}^{\infty }\Pi (\Pi B^{\otimes j})^{\mathbb{Z%
}/j\mathbb{Z}})^{\vee }
\end{equation*}%
Then I define the sum over oriented trivalent graphs parallel to (\ref{S}).
For the case of the even scalar product the variant of the previous theorem
holds. The proof is the same.

\begin{theorem}
The sum over oriented trivalent ribbon graphs $S$ satisfy the equivariant
noncommutative Batalin-Vilkovisky equation: 
\begin{equation*}
\hbar \Delta S+\frac{1}{2}\{S,S\}+I^{\vee }S=0,\,\,
\end{equation*}%
in particular if $I|_{B}=0$ then $S$ is the solution of the non-commutative
Batalin-Vilkovisky equation from \cite{B06a},\cite{B06b} 
\begin{equation*}
\hbar \Delta S+\frac{1}{2}\{S,S\}=0
\end{equation*}%
If $I|_{B}\neq 0$, but $I^{2}|_{B}=0$, then $S+S_{0,2}$ is also a solution
to the non-commutative Batalin-Vilkovisky equation from \cite{B06a},\cite%
{B06b}, where $S_{0,2}=(-1)^{\epsilon }\beta (I\cdot ,\cdot )|_{B}$ is the
quadratic term corresponding to the differential $I|_{B}$.
\end{theorem}

\subsection{General algebras.}

Let now $A$, $\dim _{k}A<\infty $ , be an arbitrary $\mathbb{Z}/2\mathbb{Z-}$%
graded algebra, with an \emph{odd differentiation} $I:A\rightarrow \Pi A$.
This case is reduced to the previous cases by putting $\widetilde{A}=A\oplus
(\Pi A)^{\vee }$ with odd scalar product $\beta $ given by natural odd
pairing between $A$ and $(\Pi A)^{\vee }$, or by putting $\widetilde{A}%
=A\oplus A^{\vee }$ with even scalar product $\beta $ given by natural even
pairing between $A$ and $A^{\vee }$. Then $\widetilde{A}$ is naturally an
associative algebra with odd, respectively even, scalar product. For any $%
a^{\ast },b^{\ast }$ from $(\Pi A)^{\vee }$, respectfully from $A^{\vee },$ 
\begin{equation*}
m_{2}(a^{\ast },b^{\ast })=0
\end{equation*}%
and $m_{2}(a^{\ast },b)$ takes value in $(\Pi A)^{\vee }$, respectfully in $%
A^{\vee }$, 
\begin{equation*}
m_{2}(a^{\ast },b)c=a^{\ast }(m_{2}(b,c))
\end{equation*}%
and similarly for $m_{2}(a,b^{\ast })$. The cyclic three-tensor describing
this cyclic associative algebra is simply the initial multplication tensor 
\begin{equation*}
m_{2}\in (\Pi A\otimes \Pi A)^{\vee }\otimes A
\end{equation*}%
considered as element of 
\begin{equation*}
((\Pi A^{\vee }\oplus A)^{\otimes 3})^{\mathbb{Z}/3\mathbb{Z}}
\end{equation*}%
or, respectfully, 
\begin{equation*}
\Pi ((\Pi A^{\vee }\oplus \Pi A)^{\otimes 3})^{\mathbb{Z}/3\mathbb{Z}}
\end{equation*}%
The \emph{odd differentiation} $I$ action extends naturally to $\widetilde{A}
$.

Consider the case of $\widetilde{A}=A\oplus (\Pi A)^{\vee }$ with its odd
scalar product, $\dim _{k}A<\infty $. Suppose that $H$ is an odd operator%
\begin{equation*}
H:A\rightarrow \Pi A,~~~
\end{equation*}%
such that 
\begin{equation}
Id-[I,H]=P  \label{idh}
\end{equation}%
is an idempotent operator $P:A\rightarrow A,$%
\begin{equation*}
~P^{2}=P.
\end{equation*}%
I assume also that $H$ commutes with $I^{2}$, this is automatic if $I^{2}=0$%
. Then both $H$ and $P$ act naturally on $\widetilde{A}$ as self-adjoint
operators and I apply to this situation the construction of tensors $%
W_{\Gamma }$ for ribbon graphs described above. The tensors 
\begin{equation*}
W_{\Gamma }^{B}\in Symm(\oplus _{j=1}^{\infty }((\Pi B^{\vee }\oplus
B)^{\otimes j})^{\mathbb{Z}/j\mathbb{Z}})
\end{equation*}%
are defined by the contraction (\ref{wg}). The contraction is given by the
sum over markings by  
\begin{equation*}
Flag(\Gamma )\rightarrow \{\Pi A,A^{\vee }\}
\end{equation*}%
such that for any edge, its two flags are marked differently, and for any
vertex there is exactly one flag which is marked by $\Pi A$, with no other
extra restrictions. In particular such marking gives an orientation for
every edge, from $A^{\vee }$ to $\Pi A$, and there must be exactly one edge
exiting every vertex. The legs of $\Gamma $, which correspond to the points
sitting on the boundary of the surface $\Sigma _{\Gamma }$, are also marked
as either entries ($B^{\vee }$) or exits ($\Pi B$). And I define $S^{B}$ by
the summation as above 
\begin{equation*}
S^{B}=\tsum_{\{\Gamma \}}\hbar ^{1-\chi (\Sigma _{\Gamma })}W_{\Gamma }^{B}
\end{equation*}%
where the sum is over isomorphism classes of connected trivalent graphs with
such orientation on edges and with nonempty subsets of legs on every
boundary component of $\Sigma _{\Gamma }$.

Similarly I define the tensors $W_{\Gamma }$ in the case of $\widetilde{A}%
=A\oplus A^{\vee }$ with its even pairing. These tensors belong to the space
of exterior powers of linear functionals on cyclic words on elements of $\Pi
B\ $ and $\Pi B^{\vee }$ 
\begin{equation*}
W_{\Gamma }^{\Pi B}\in Symm(\oplus _{j=1}^{\infty }\Pi ((\Pi B\oplus \Pi
B^{\vee })^{\otimes j})^{\mathbb{Z}/j\mathbb{Z}})
\end{equation*}%
and I define $S^{\Pi B}$as their sum over oriented ribbon graphs as above.

\begin{theorem}
Let $A$ be an arbitrary $\mathbb{Z}/2\mathbb{Z-}$graded algebra, $\dim
_{k}A<\infty $ , with an \emph{odd differentiation} $I:A\rightarrow \Pi A$
and a homotopy $H$, for which the operator (\ref{idh}) is idempotent. The
sums over ribbon graphs $S^{B}$ and $S^{\Pi B}$ give the solutions to the
two variants of the equivariant noncommutative Batalin-Vilkovisky equation
in the spaces of symmetric, respectfully exterior powers, of cyclic words,
consisting of elements from $\Pi B^{\vee }$ and $B$, respectfully from $\Pi
B^{\vee }$ and $\Pi B$:
\end{theorem}

\begin{eqnarray*}
\hbar \Delta S^{B}+\frac{1}{2}\{S^{B},S^{B}\}+I^{\vee }S^{B} &=&0 \\
\hbar \Delta S^{\Pi B}+\frac{1}{2}\{S^{\Pi B},S^{\Pi B}\}+I^{\vee }S^{\Pi B}
&=&0.
\end{eqnarray*}

\begin{proof}
This is an immediate corollary of the theorem from the previous subsection
for the algebras with odd/ even invariant scalar product $\widetilde{A}%
=A\oplus (\Pi A)^{\vee }$ or $\widetilde{A}=A\oplus A^{\vee }$.
\end{proof}

\section{Graphs with the insertion of $A_{\infty }-$tensors.}

Assume now that $A$ is a $\mathbb{Z}/2\mathbb{Z-}$graded $A_{\infty }-$%
algebra, $\dim _{k}A<\infty $ . I relax, as above, the condition of the
square of differential equals to zero, and assume that it is simply an \emph{%
odd operator} $I:A\rightarrow \Pi A$, which together with other structure
maps $m_{n}\in ((\Pi A)^{\otimes n})^{\vee }\otimes A$, $n\geq 2$, satisfy
the standard $A_{\infty }-$constrains, except perhaps the very first, so
that, in general $I^{2}\neq 0$ : for any $n\geq 2$%
\begin{multline}
Im_{n}(v_{1},\ldots ,v_{n})-\tsum_{l}(-1)^{\epsilon }m_{n}(v_{1},\ldots
,Iv_{l},\ldots v_{n})=  \label{imn} \\
=\tsum_{i+j=n+1}(-1)^{\epsilon }m_{i}(v_{1},\ldots ,m_{j}(\ldots ),\ldots
v_{n})  \notag
\end{multline}%
or, equivalently, 
\begin{equation*}
I^{\vee }m+\{m,m\}=0
\end{equation*}

I assume first that $A$ has also an invariant odd scalar product $\beta $ so
that all tensors 
\begin{equation*}
m_{n}\in ((\Pi A)^{\otimes n+1})^{\vee },\beta (~m_{n}(v_{1},\ldots
,v_{n}),v_{n+1})
\end{equation*}%
are cyclic invariant, the variant without scalar product is reduced as above
to this case by taking $\widetilde{A}=A\oplus (\Pi A)^{\vee }$, see below.

Let as above $H$ be an odd selfadjoint operator%
\begin{equation*}
H:A\rightarrow \Pi A,~~~H^{\vee }=H
\end{equation*}%
such that 
\begin{equation}
Id-[I,H]=P  \label{idhh}
\end{equation}%
is an idempotent operator $P:A\rightarrow A$, whose image I denote by $B$. I
assume also as above that $H$ commutes with $I^{2}$, this is of course
automatic if $I^{2}=0$.

Now I define the tensors $W_{\Gamma }$, by inserting the cyclyc tensors $%
m_{n(v)}\in $ $((\Pi A)^{\otimes Flag(v)})^{\vee }$ at vertices, as above,
where $\Gamma $ is now a ribbon graph, with valency $n(v)$ for any vertice
at least three:%
\begin{equation*}
W_{\Gamma }(\tbigotimes_{l\in Leg(\Gamma )}a_{l})=\left\langle
\tbigotimes_{v\in Vert(\Gamma )}m_{n(v)},\left( \tbigotimes_{e\in
Edge(\Gamma )}\beta _{H}^{\vee ,e}\right) \tbigotimes_{l\in Leg(\Gamma
)}a_{l}\right\rangle
\end{equation*}%
and 
\begin{equation*}
W_{\Gamma }\in Symm(\oplus _{j=1}^{\infty }(\Pi B^{\otimes j})^{\mathbb{Z}/j%
\mathbb{Z}})^{\vee }
\end{equation*}%
The following is the standard lemma, which ensures that the sum over ribbon
graphs is actually finite at each order of $\hbar $.

\begin{proposition}
The number of ribbon graphs, with valency at every vertex $n(v)\geq 3$, with
fixed $\chi (\Sigma _{\Gamma })$ and fixed number of exterior legs is finite.
\end{proposition}

\begin{proof}
Similarly to above, if I denote the number of vertices of valency $n$ via $%
v_{n}$: 
\begin{equation*}
\chi (\Sigma _{\Gamma })+|Edge(\Gamma )|=\tsum v_{n}
\end{equation*}%
and 
\begin{equation*}
n(\Gamma )+2|Edge(\Gamma )|=\tsum nv_{n}.
\end{equation*}%
It follows that 
\begin{equation*}
\tsum (n-2)v_{n}=n(\Gamma )-2\chi (\Sigma _{\Gamma })
\end{equation*}%
and hence $\tsum v_{n}\leq const$ and $|Edges(\Gamma )|\leq const$.
\end{proof}

At the next step however, looking carefully at the proof of the equation for 
$S$ above, one sees that one immediately runs into a problem because of
tadpoles, i.e. self-contractions of nonneighbouring flags at the same
vertex, unless the following important condition 
\begin{equation*}
\Delta m_{n}=0
\end{equation*}%
is imposed, which I assume from now till the end of this section.

I define next, similarly to above,%
\begin{equation}
S=\tsum_{\{\Gamma \}}\hbar ^{1-\chi (\Sigma _{\Gamma })}W_{\Gamma }
\label{SA}
\end{equation}%
where the sum is over isomorphism classes of connected ribbon graphs with
vertices of valency $n(v)\geq 3$, and with nonempty subsets of legs on every
boundary component of $\Sigma _{\Gamma }$.

\begin{theorem}
Let the odd operator $I$ and the cyclically invariant tensors $m_{n}\in
((\Pi A)^{\otimes n+1})^{\vee }$, $n\geq 2$, satisfy 
\begin{eqnarray*}
I^{\vee }m+\{m,m\} &=&0 \\
\Delta m &=&0
\end{eqnarray*}%
Then, given the homotopy $H$ (\ref{idhh}), the sum $S$ (\ref{SA}) over
ribbon graphs satisfy the equivariant noncommutative Batalin-Vilkovisky
equation associated with $(B,\beta |_{B})$: 
\begin{equation*}
\hbar \Delta S+\frac{1}{2}\{S,S\}+I^{\vee }S=0,\,\,
\end{equation*}%
in particular if $I|_{B}=0$ then $S$ is the solution of the non-commutative
Batalin-Vilkovisky equation from \cite{B06a},\cite{B06b} 
\begin{equation*}
\hbar \Delta S+\frac{1}{2}\{S,S\}=0
\end{equation*}%
If $I|_{B}\neq 0$, but $I^{2}|_{B}=0$, then $S+S_{0,2}$ is also a solution
to the non-commutative Batalin-Vilkovisky equation from \cite{B06a},\cite%
{B06b}, where $S_{0,2}=(-1)^{\epsilon }\beta (I\cdot ,\cdot )|_{B}$ is the
quadratic term corresponding to the differential $I|_{B}$.
\end{theorem}

\begin{proof}
The proof is parallel to the above.
\end{proof}

As above the case of algebra with no scalar product is reduced to the cyclic
algebra case.

\begin{proposition}
Same result holds in the context of arbitrary $\mathbb{Z}/2\mathbb{Z-}$%
graded equivariant $A_{\infty }-$algebra, $\dim _{k}A<\infty $ , with odd
differentiation $I:A\rightarrow \Pi A$, 
\begin{equation*}
I^{\vee }m+\{m,m\}=0
\end{equation*}%
As above the algebra must satisfy in addition the condition $\Delta m=0$,
satisfied trivially by the associative algebras. This case is reduced to the
previous cases by putting $\widetilde{A}=A\oplus (\Pi A)^{\vee }$ with odd
scalar product $\beta $ given by natural odd pairing between $A$ and $(\Pi
A)^{\vee }$, or by putting $\widetilde{A}=A\oplus A^{\vee }$ with even
scalar product $\beta $ given by natural even pairing between $A$ and $%
A^{\vee }$.
\end{proposition}

\begin{theorem}
Given arbitrary solution to the equivariant non-commutative
Batalin-Vilkovisky equation on $\mathbb{Z}/2\mathbb{Z-}$graded vector space $%
A$, $\dim _{k}A<\infty $, with odd scalar product $\beta $, 
\begin{gather}
\widehat{m}\in Symm(\oplus _{j=1}^{\infty }(\Pi A^{\otimes j})^{\mathbb{Z}/j%
\mathbb{Z}})^{\vee }[[\hbar ]], \\
\hbar \Delta \widehat{m}+\frac{1}{2}\{\widehat{m},\widehat{m}\}+I^{\vee }%
\widehat{m}=0,\,\,  \notag
\end{gather}
and a homotopy $H$ (\ref{idhh}), I define the tensors $W_{\Gamma }$ for 
\emph{stable} ribbon graphs by contraction as above and the sum $S$ over
such \emph{stable} ribbon graphs, then $S$ is again a solution to the
equivariant non-commutative Batalin-Vilkovisky equation on $(B,\beta |_{B})$:%
\begin{equation*}
\hbar \Delta S+\frac{1}{2}\{S,S\}+I^{\vee }S=0
\end{equation*}%
The same result holds in the case of the even scalar product.

\begin{proof}
Analogous to the above, see also the general case of algebras over modular
operad in the next section.
\end{proof}
\end{theorem}

\section{Construction of solutions to the $\mathcal{P}$-Batalin-Vilkovisky
equation.\label{sectmodop}}

Let now $A=\oplus _{i}A^{i}$ be a $\mathbb{Z}-$graded vector space with the
odd scalar product $\beta $ of degree $2l+1$, and let $A$ is an algebra over 
$\mathcal{FP}$, the Feynman transform of a modular operad $\mathcal{P}$. I
consider for simplicity the $\mathbb{Z}-$graded version as defined in \cite%
{GK}, see also \cite{B06a}. Without loss of generality one can assume that $%
l=1,$the general case is reduced to this case by a twisting by a cocycle.
The parallel results hold in the $\mathbb{Z}/2\mathbb{Z}-$graded setting and
also for even scalar products, I leave details to an interested reader.

By theorem 1 from \cite{B06a}, an algebra over $\mathcal{FP}$ is defined by
set of elements 
\begin{equation*}
m_{n,b}\in \left( \left( (A{}[1]\right) ^{\otimes n})^{\vee }\otimes 
\mathcal{P}((n,b))\right) ^{\mathbb{S}_{n}}
\end{equation*}%
with $\widehat{m}=\sum_{n,b}\hbar ^{b}m_{n,b}$ satisfying the
Batalin-Vilkovisky equation of $\mathcal{P}-$geometry, the equation (5.5)
from \cite{B06a}.

As above I consider more general notion, which is natural to call
"equivariant algebra over $\mathcal{FP}$". This is the $\mathbb{Z}$-graded
vector space $A$ with scalar product $\beta $, with a degree $1$\emph{\ }%
anti-selfadjoint\emph{\ }operator $I:A\rightarrow A[1]$,and provided with
the morphism from the free $\mathcal{K}-$twisted modular operad 
\begin{equation*}
\Phi :\mathbb{M}_{\mathcal{K}}\mathcal{P}^{dual}\rightarrow \mathcal{E}[A]
\end{equation*}%
equivariant with respect to the odd differentiations 
\begin{equation*}
I\circ \Phi =\Phi \circ d_{\mathcal{FP}}.
\end{equation*}%
This is equivalent for $\widehat{m}$ to satisfy the equivariant version of
the $\mathcal{P}-$Batalin-Vilkovisky equation: 
\begin{equation}
\hbar \Delta \widehat{m}+\frac{1}{2}\{\widehat{m},\widehat{m}\}+I^{\vee }%
\widehat{m}=0,~~~m_{n,b}\in \left( \left( (A{}[1]\right) ^{\otimes n})^{\vee
}\otimes \mathcal{P}((n,b))\right) ^{\mathbb{S}_{n}}  \label{pBV}
\end{equation}%
where $\Delta $ and $\{\cdot ,\cdot \}$ are defined in \cite{B06a}, .

Let as above $H$ be a homotopy, degree $-1$, selfadjoint operator%
\begin{equation*}
H:A\rightarrow A[-1],~~~H^{\vee }=H
\end{equation*}%
such that 
\begin{equation*}
Id-[I,H]=P
\end{equation*}%
is an idempotent operator $P:A\rightarrow A$, whose image I denote by $B$.
Assume as above that $H$ commutes with $I^{2}$, this is automatic if $I^{2}=0
$. Next I define the summation over $\mathcal{P}-$decorated graphs with
legs, i.e. graphs with decorations from $\mathcal{P}((Flag(v)))$ attached to
vertices. For a stable graph $\Gamma $ I put

\begin{itemize}
\item \bigskip the tensors 
\begin{equation*}
m^{v}\in \left( \left( (A{}[1]\right) {}^{\otimes Flag(v)})^{\vee }\otimes 
\mathcal{P}((Flag(v),b(v)))\right) ^{Aut(Flag(v))}
\end{equation*}%
on every vertex $v$, obtained from $m_{n(v),b(v)}$ using the functor of
extension to finite sets,

\item the two tensors 
\begin{gather*}
\beta _{H}^{\vee ,e}\in (A[1])^{\otimes \{f,f^{\prime }\}}, \\
\beta _{H}^{\vee ,e}=\beta ^{\vee }(H^{\vee }u_{f},v_{f^{\prime }})=(-1)^{%
\overline{u_{f}}~\overline{v_{f^{\prime }}}}\beta ^{\vee }(H^{\vee
}v_{f^{\prime }},u_{f})
\end{gather*}%
for any interieur edge $e=(ff^{\prime })$

\item element $a_{l}\in B[1]$, for any leg $l$.
\end{itemize}

Notice that both $m^{v}$ and $\beta _{H}^{\vee ,e}$ are degree zero
elements, so that the products $\tbigotimes_{v\in Vert(\Gamma )}m^{v}$ and $%
\tbigotimes_{e\in Edge(\Gamma )}\beta _{H}^{\vee ,e}$ are canonically
defined. Recall that for the modular operad $\mathcal{P}$, for any stable
graph $\Gamma $ a bijection $Leg(\Gamma )\leftrightarrow \{1,\ldots n\}$
there is the composition map%
\begin{equation*}
\mu _{\Gamma }:\tbigotimes_{v\in Vert(\Gamma )}\mathcal{P}%
((Flag(v),b(v)))\rightarrow \mathcal{P}((n(\Gamma ),b(\Gamma )))
\end{equation*}

\begin{definition}
I define the tensor 
\begin{equation*}
W_{\Gamma }\in \left( ((B[1]){}^{\otimes n(\Gamma )})^{\vee }\otimes 
\mathcal{P}((n(\Gamma ),b(\Gamma )))\right) ^{\mathbb{S}_{n(\Gamma )}}
\end{equation*}%
as the contraction of tensors on $A[1]$ times the $\mathcal{P}-$composition
structure map $\mu _{\Gamma }$ 
\begin{equation}
W_{\Gamma }(\tbigotimes_{l\in Leg(\Gamma )}a_{l})=\left\langle
\tbigotimes_{v\in Vert(\Gamma )}m^{v},\left( \tbigotimes_{e\in Edge(\Gamma
)}\beta _{H}^{\vee ,e}\right) \tbigotimes_{l\in Leg(\Gamma
)}a_{l}\right\rangle \otimes \mu _{\Gamma }  \label{Wp}
\end{equation}
\end{definition}

For a given graph $\Gamma $ and choice of basis in every $\mathcal{P}%
((Flag(v),b(v)))$ this expresion is the sum, over markings of vertices of $%
\Gamma $ by basis elements of $\mathcal{P}$, of the corresponding
contractions of tensors on $A[1]$.

\begin{definition}
The sum over $\mathcal{P}$-marked graphs is defined by 
\begin{equation}
S_{n,b}=\sum_{\Gamma \in \lbrack G((n,b))]}W_{\Gamma }  \label{Sp}
\end{equation}%
where $[G((n,b))]$ denotes the set of isomorphisms classes of pairs $(\Gamma
,\rho )$ where $\Gamma $ is a stable graph with $n(\Gamma )=n$, $b(\Gamma
)=b $ and $\rho $ is a bijection $Leg(\Gamma )\leftrightarrow \{1,\ldots n\}$%
. I put 
\begin{equation*}
S=\sum_{n,b}\hbar ^{b}S_{n,b}
\end{equation*}
\end{definition}

The sum in the definition of $S_{n,b}$ is finite, see \cite{GK} lemma 2.16.

\begin{theorem}
The series $S$, given by the sum over $\mathcal{P}$-marked graphs, satisfy
the equivariant Batalin-Vilkovisky equation associated with the modular
operad $\mathcal{P}$: 
\begin{equation*}
\hbar \Delta S+\frac{1}{2}\{S,S\}+I^{\vee }S=0,\,\,
\end{equation*}%
in particular if $I^{2}|_{B}=0$ then $S$ is the solution of the
Batalin-Vilkovisky equation, associated with the modular operad $\mathcal{P}$%
, from \cite{B06a}: 
\begin{equation*}
\hbar \Delta S+\frac{1}{2}\{S,S\}+dS=0
\end{equation*}%
where I denoted by $d$ the differential $I^{\vee }|_{B^{\vee }}$. By the
theorem 1 of \cite{B06a} , this is equivalent to the fact, that $S$ defines
on $B$ the structure of algebra over $\mathcal{FP}$, the Feynman transform
of the modular operad $\mathcal{P}$.
\end{theorem}

\begin{proof}
The proof is parallel to the proof of the theorem \ref{th1} above, via
introducing the tensors $W_{\Gamma ,e}^{[I,H]},W_{\Gamma ,e}^{Id},W_{\Gamma
,e}^{P}$ and verifying that they satisfy analogous relations. One has to use
the definition of the $\mathcal{P}$-type Batalin-Vilkovisky operator $\Delta 
$ from \cite{B06a}.
\end{proof}

\begin{remark}
The particular case of this statement, in its nonequivariant version with $%
I^{2}=0$, applied for the modular extension of the $L_{\infty }-$operad,
gives the transfer of  solutions to the ordinary (commutative)
Batalin-Vilkovisky equation. In the case  $\widetilde{A}=A\oplus
(A[1])^{\vee }$ one gets the solutions via the summation over graphs with
oriented edges. A similar statement which starts from solutions of degree $%
\leq 3$ to the commutative Batalin-Vilkovisky equation, satisfying certain
extra boundary conditions, and in which the sum is taken over the subset of
"directed" graphs of the set of oriented graphs, is established in \cite{M08}%
, together with some generalisation. 
\end{remark}

\begin{remark}
Any cyclic operad can be considered as modular with zero selfcontractions.
Hence this theorem gives also the analogous result for algebras over the Bar
-transform of cyclic operads, and also, via setting  $\widetilde{A}=A\oplus
A^{\vee }$ for  ordinary operads, see \cite{H}, \cite{M08} and references
therein. The equivariant version, with the odd derivation relaxing the
condition on the differential $d^{2}\neq 0$, is new even in the standard
case of ordinary/cyclic operads.
\end{remark}

\begin{proposition}
For any solution to the \emph{equivariant} Batalin-Vilkovisky equation, the
construction of the homology class of the associated graph complex of $%
\mathcal{P}$-marked graphs from \cite{B06a} works without change and gives
the homology class of the complex dual to $\mathcal{FP}|_{n(\Gamma )=0}$.
\end{proposition}

\begin{proof}
This is immediate from the definition, the $n(\Gamma )=0$ part of the map $%
\Phi $ satisfies%
\begin{equation*}
\Phi _{n(\Gamma )=0}(d_{\mathcal{FP}}(\alpha ))=I(\Phi _{n(\Gamma
)=0}(\alpha ))=0
\end{equation*}%
since the odd operator $I$ acts by zero on $k=\mathcal{E}[A]((0,b))$.
\end{proof}

\begin{theorem}
Let $\widehat{m}$ be a solution to the equivariant Batalin-Vilkovisky
equation of $\mathcal{P}-$geometry (\ref{pBV}) and let $H$ be a self-adjoint
nilpotent homotopy as above. Then the solution $S$ to (\ref{pBV}) obtained
by the summation over $\mathcal{P}$-marked graphs (\ref{Wp}), (\ref{Sp}),
and $\widehat{m}$ have the same characteristic homology class in the graph
complex of $\mathcal{P}$-marked graphs.$\square $
\end{theorem}

\section{$A$-infinity $gl(N)-$equivariant matrix integrals.\label%
{sectionIntegr}}

Here I briefly describe the $A$-infinity $gl(N)-$equivariant matrix
integrals which I've introduced in \cite{B06b}. I focus on the
odd-dimensional case. I recall in particular the results from \cite{B09a},
that establish the correspondence between solutions to the noncommutative
Batalin-Vilkovisky equation and equivariantly closed differential forms on $%
\left( gl(N|N)\otimes \Pi V\right) _{0}$.

Let $V$ be a $\mathbb{Z}/2\mathbb{Z}$-graded vector space, $\dim
_{k}V<\infty $ , with \emph{odd }scalar product $\beta :V^{\otimes
2}\rightarrow \Pi k$. Let 
\begin{gather}
m_{A_{\infty }}\in \tbigoplus_{j}(((\Pi V)^{\otimes j})^{\mathbb{Z}/j\mathbb{%
Z}})^{\vee },~ \\
\{m_{A_{\infty }},m_{A_{\infty }}\}=0  \label{mmA}
\end{gather}%
be the cyclic tensor defining the structure of cyclic $A_{\infty }-$algebra
on $A$, with the invariant odd scalar product $\beta $. To any product of
cyclic words from 
\begin{equation*}
Symm(\oplus _{j=1}^{\infty }((\Pi V\oplus \Pi k_{\Lambda })^{\otimes j})^{%
\mathbb{Z}/j\mathbb{Z}})^{\vee }
\end{equation*}%
I associate, using the invariant theory, the invariant functions, which are
product of traces, from%
\begin{equation}
\left( \mathcal{O}(gl(N|N)\otimes \Pi V)\otimes \mathcal{O}(\Pi
gl(N|N))\right) ^{gl(N|N)}  \label{oglnv}
\end{equation}%
In particular $Tr(m_{A_{\infty }})$ denotes the function on $gl(N|N)\otimes
\Pi V$ 
\begin{equation*}
\tsum_{(i_{1}\ldots i_{k})}(-1)^{\epsilon }m_{A_{\infty }},_{i_{1}\ldots
i_{k}}tr(X^{i_{1}}\cdot \ldots \cdot X^{i_{k}}),~~~X=X^{i}\otimes \pi
e_{i},~X^{i}\in gl(N|N),~e_{i}\in V
\end{equation*}%
extending the cyclic $A_{\infty }$ tensor $m_{A_{\infty }}$. I've introduced
in \cite{B06b} integrals

\begin{equation}
\int_{\gamma \in \left( gl(N|N)\otimes \Pi V\right) _{0}}\exp \frac{1}{\hbar 
}(s_{\Lambda }+Tr(m_{A_{\infty }})+\sum_{2g+i>2}\hbar
^{2g+i-1}S_{i,g}^{mTr})\varphi ~dX  \label{expmaplus}
\end{equation}%
where 
\begin{equation*}
s_{\Lambda }=\tsum_{i_{1}i_{2}}(-1)^{\epsilon }\beta _{i_{1}i_{2}}tr([\Xi
,X^{i_{1}}],X^{i_{2}}),
\end{equation*}%
is the hamiltonian of action of  odd matrix $\Xi \in gl(N|N)_{1}$, which in
generic case without loss of generality can be assumed to be of the form $%
\Xi =\left( 
\begin{array}{cc}
0 & Id \\ 
\Lambda  & 0%
\end{array}%
\right) $, $\Lambda \in $ $gl(N).$ And $S_{i,g}^{mTr}$ are the multitrace
elements corresponding to products of cyclic words 
\begin{equation*}
S_{i,g}\in Symm^{i}(\oplus _{j=1}^{\infty }((\Pi V)^{\otimes j})^{\mathbb{Z}%
/j\mathbb{Z}})^{\vee }
\end{equation*}%
Functions on the odd symplectic affine manifold $gl(N|N)\otimes \Pi V$ are
identified canonically with polyvector fields on $\left( gl(N|N)\otimes \Pi
V\right) _{0}$ which are in turn identified with differential forms on the
same space, using a constant in the affine coordinates holomorphic volume
form $dX$. The invariant functions from (\ref{oglnv}) are then mapped to $%
gl(N)$- equivariant differential forms, with respect to $gl(N)$ acting by
conjugation via block-diagonal embedding. Let $d_{DR}-i_{\Lambda }$ denotes
the $gl(N)$-equivariant differential.

\begin{proposition}
\begin{enumerate}
\item \label{propintegr} The lagrangian $S=m_{A_{\infty
}}+\sum_{2g+i>2}\hbar ^{2g+i-1}S_{i,g}$ satisfies the noncommutative
Batalin-Vilkovisky equation (\ref{ncBV}) if and only if the  $gl(N)-$%
equivariant differential form corresponding to $s_{\Lambda }+S$ is closed
for any $N$%
\begin{equation*}
\left( d_{dR}-i_{\Lambda }\right) \left( \exp \frac{1}{\hbar }(s_{\Lambda
}+Tr(m_{A_{\infty }})+\sum_{2g+i>2}\hbar ^{2g+i-1}S_{i,g}^{mTr})\vdash
dX\right) =0
\end{equation*}

\item Similarly, for any 
\begin{equation*}
\varphi =\sum_{i,g\geq 0}\hbar ^{2g+i-1}\varphi _{i,g},~~\varphi _{i,g}\in
Symm^{i}(\oplus _{j=1}^{\infty }((\Pi B)^{\otimes j})^{\mathbb{Z}/j\mathbb{Z}%
})^{\vee },
\end{equation*}%
\begin{gather*}
\hbar \Delta \varphi +\frac{1}{2}\{S,\varphi \}=0\Longleftrightarrow  \\
\left( d_{dR}-i_{\Lambda }\right) \left( \varphi ^{mTr}\exp \frac{1}{\hbar }%
(s_{\Lambda }+S^{mTr})\vdash dX\right) =0
\end{gather*}

\item  For 
\begin{equation*}
f\in \left( \mathcal{O}(gl(N|N)\otimes \Pi (V\oplus k_{\Lambda })\right)
^{gl(N|N)}(\hbar )
\end{equation*}
the $gl(N)-$equivariant differential form corresponding to $s_{\Lambda }+f$
is closed  
\begin{equation*}
\left( d_{dR}-i_{\Lambda }\right) \left( \exp \frac{1}{\hbar }(s_{\Lambda
}+f)\vdash dX\right) =0
\end{equation*}%
if and only if $f$ satisfies the noncommutative equivariant
Batalin-Vilkovisky equation 
\begin{equation*}
\hbar \Delta f+\frac{1}{2}\{f,f\}+I^{\vee }f=0,\,\,
\end{equation*}%
where $I=[\Xi ,\cdot ]$
\end{enumerate}
\end{proposition}

\begin{proof}
The proof is immediate from the results of \cite{B09a}.
\end{proof}

\begin{remark}
The algebra $gl(N|N)$ with its even trace, can be replaced without any other
change in the formalism, by any finite dimensional super associative algebra 
$g$ with trace, which satisfy $tr_{g}(l_{a})=0$, where $l_{a}$ is the
operator of multiplication by the arbitrary element $a$. The most part of
the formalism works without change also for any finite dimensional super
associative algebra with trace.
\end{remark}

\subsection{Equivariant matrix integrals for associative algebras and
intersections on moduli spaces of curves.}

\begin{proposition}
Let $A$, $\dim _{k}A<\infty $, be associative d($\mathbb{Z}/2\mathbb{Z}$%
)g-algebra with \emph{odd }scalar product $\beta $, multiplication described
by cyclic 3-tensor $m_{2}$. Then $\Delta (m_{2})=0$ and%
\begin{equation*}
S=\left( -\frac{1}{2}\tsum_{i_{1}i_{2}}(-1)^{\epsilon }\beta
_{i_{1}i_{2}}tr([\Xi ,X^{i_{1}}],X^{i_{2}})+\frac{1}{3!}%
\tsum_{i_{1}i_{2}i_{3}}(-1)^{\epsilon
}m_{2,i_{1}i_{2}i_{3}}tr(X^{i_{1}}X^{i_{2}}X^{i_{3}})\right)
\end{equation*}%
defines closed $gl(N)$-equivariant differential form on $(gl(N|N)\otimes \Pi
V)_{0}$. It can be seen as an odd higher dimensional generalisation of the
matrix Airy integral. It's asymptotic expansion is given, as it folows from
theorem 1 of \cite{B06b}), by the sum over trivalent ribbon graphs :%
\begin{equation*}
\exp \left( const\sum_{\Gamma }\hbar ^{-\chi _{\Gamma }}c_{S}(\Gamma
)c_{\Lambda }(\Gamma )\right)
\end{equation*}%
where $c_{\Lambda }\in H^{\ast }(\overline{\mathcal{M}}_{g,n}),$ $c_{S}\in
H_{\ast }(\overline{\mathcal{M}}_{g,n})$ are the cocycle and the cycle on
the stable ribbon graph complex defined for any stable ribbon graph in \cite%
{B06b}, \cite{B09b} and associated with the odd differentiation $I=[\Lambda
_{01},\cdot ]$, $I^{2}\neq 0$, of the associative algebra $gl(N|N)$ and with
the solution $S$ to the noncommutative Batalin-Vilkovisky equation.$\square $
\end{proposition}

\end{document}